\documentclass[12pt]{amsart} 
\usepackage{amssymb} 
\usepackage{graphicx} 
\usepackage{colordvi} 
 
\numberwithin{equation}{section} 
\newtheorem{theorem}{Theorem}[section] 
\newtheorem{proposition}[theorem]{Proposition} 
\newtheorem{lemma}[theorem]{Lemma} 
\newtheorem{corollary}[theorem]{Corollary} 
\theoremstyle{remark} 
\newtheorem{example}[theorem]{Example} 
\newtheorem{remark}[theorem]{Remark}

\newcounter{FNC}[page] 
\def\fauxfootnote#1{{\addtocounter{FNC}{2}$^\fnsymbol{FNC}$%
     \let\thefootnote\relax\footnotetext{$^\fnsymbol{FNC}$#1}}}
\newcommand{\defcolor}[1]{\Blue{#1}} 
\newcommand{\demph}[1]{\defcolor{{\sl #1}}}

\newcommand{\pt}{{\it pt}}

\renewcommand{\P}{{\mathbb{P}}} 
\newcommand{\C}{{\mathbb{C}}} 
 
\newcommand{\Q}{{\mathbb{Q}}} 
\newcommand{\Z}{{\mathbb{Z}}} 

\newcommand{\calF}{{\mathcal{F}}} 
\newcommand{\calS}{{\mathcal{S}}} 
\newcommand{\calI}{{\mathcal{I}}} 
\newcommand{\calL}{{\mathcal{L}}} 
\newcommand{\calO}{{\mathcal{O}}} 
 
\newcommand{\frakC}{{\mathfrak{C}}}

\newcommand{\frakS}{{\mathfrak{S}}} 

\DeclareMathOperator{\Maps}{Maps}
\DeclareMathOperator{\diag}{diag}

\title[Equivariant cohomology theories and the pattern map]{Equivariant cohomology theories\\ and the
  pattern map}  
 
\author{Praise Adeyemo} 
\address{Department of Mathematics\\ 
  University of Ibadan\\ 
   Ibadan, Oyo, Nigeria} 
\email{ph.adeyemo@ui.edu.ng} 
\urladdr{http://sci.ui.edu.ng/HPAdeyemo} 
 
\author{Frank Sottile} 
\address{Department of Mathematics\\ 
         Texas A\&M University\\ 
         College Station\\ 
         Texas\\ 
         USA} 
\email{sottile@math.tamu.edu} 
\urladdr{http://www.math.tamu.edu/\~{}sottile}

\thanks{Adeyemo was partially supported by a grant from the Niels Henrik Abel Board of the
  IMU, and Sottile was partially supported by NSF grant DMS 15-1501370.}
 
\begin{document} 
 
\begin{abstract} 
 Billey and Braden defined a geometric pattern map on flag manifolds which extends the 
 generalized pattern map of Billey and Postnikov on Weyl groups. 
 The interaction of this torus equivariant map with the Bruhat order and its action on line bundles lead to 
 formulas for its pullback on the equivariant cohomology ring and on equivariant $K$-theory. 
 These formulas are in terms of the Borel presentation, the basis of Schubert classes, and localization at torus
 fixed points.
\end{abstract} 
 
\maketitle 
%
\section*{Introduction} 
 
The (geometric) pattern map of Billey and Braden~\cite{BB} is a map between flag manifolds that extends the
generalized pattern map on Weyl groups of Billey and Postnikov~\cite{BiPo}.
We previously studied maps on cohomology and $K$-theory induced by sections of the pattern map~\cite{AS},
generalizing formulas for specializations of Schubert and Grothendieck polynomials that had been obtained in type
A~\cite{BS98,LRS}, which are generalizations of the decomposition formula~\cite{BKTY1,BKTY2}.

The pattern map is torus-equivariant; we extend the results of~\cite{AS} to (torus) equivariant
cohomology and equivariant $K$-theory.
Specifically, we give formulas for the pattern map on equivariant cohomology and equivariant $K$-theory in
terms of localization, the Borel presentation, and Schubert classes.
The localization formula is simply restriction, while the formula for the Borel construction involves the
action of a minimal right coset representative.
Most interesting is the formula for the pullback of a Schubert class.
This is a positive (in the sense of Graham~\cite{Gr} and of Anderson, Griffeth, and Miller~\cite{AGM})
sum of Schubert classes, with coefficients certain explicit Schubert structure constants.
The value of these formulas is the interplay between them, which we illustrate through
examples.
The pattern map, together with Pieri-type formulas~\cite{So96,LS07}, was used in~\cite{BS_Skew,LRS} to obtain
new formulas for (non-equivariant) Schubert classes in type $A$ cohomology and $K$-theory.
The results of this paper should give similar formulas for equivariant Schubert classes, given a suitable
Pieri-type formula in type $A$.

In Section~\ref{S:flag}, we provide some background on equivariant cohomology and $K$-theory of flag manifolds and
explain the localization presentation of equivariant cohomology and $K$-theory.
In Section~\ref{S:pattern} we describe the pattern map and develop its interaction with localization.
We use this to prove our main results for equivariant cohomology in Section~\ref{S:equivariant}
and for equivariant $K$-theory in Section~\ref{S:K-theory}.
Section~\ref{S:Ex} contains two examples.

%
\section{Equivariant cohomology and $K$-theory of flag manifolds}\label{S:flag} 

We state standard results on equivariant cohomology and $K$-theory for flag varieties.

\subsection{Equivariant cohomology and homology}
Let $Y$ be a \demph{$T$-variety}---a projective variety equipped with a left action of a torus $\defcolor{T}\simeq
(\C^\times)^n$. 
We define $T$-equivariant homology and $T$-equivariant cohomology of $Y$ via the Borel construction.
Let $T\hookrightarrow ET\twoheadrightarrow BT$ be the universal $T$-bundle where $T$ acts freely on the right of
a contractible space $ET$ with quotient the classifying space $BT$ of $T$.
The $T$-equivariant cohomology of $Y$ is
\[
     \defcolor{H^*_T(Y)}\ :=\ H^*(ET\times_T Y, \Q)\,.
\]
We use $T$-equivariant Borel-Moore homology \defcolor{$H^T_*(Y)$}, where $ET$ is replaced by
Totaro's~\cite{Tot} sequence of finite approximation; see~\cite{Br00,EG} for details.
While $H^*_T(Y)$ is graded by the positive integers, $H^T_*(Y)$ is $\Z$-graded.
When $Y$ is a point, \defcolor{$\pt$}, $H^*_T(\pt)=H^*(BT,\Q)$, which is 
the symmetric algebra \defcolor{$S$} (over $\Q$) of the character group
$\defcolor{\Xi(T)}:=\mbox{Hom}(T,\C^\times)$ of $T$, where a character has homological degree 2. 
The map $\defcolor{\rho}\colon Y\to\pt$ to a point gives functorial maps 
$\rho^*\colon H^*_T(\pt)=S \to H^*_T(Y)$ and $\rho_*\colon H^T_*(Y) \to H^T_*(\pt)$.

Equivariant cohomology $H^*_T(Y)$ has a natural ring structure and is an $S$-algebra through the map
$\rho^*$. 
The cap product, \defcolor{$\frown$}, realizes $H^T_*(Y)$ a module over $H^*_T(Y)$ with
\[
   H^a_T(Y) \otimes H^T_b(Y)\ \xrightarrow{\ \frown\ }\ H^T_{b-a}(Y)\,,
\]
and thus $H^T_*(Y)$ is also an $S$-module. 

A $T$-invariant subvariety $Z$ of $Y$ has an equivariant fundamental cycle, 
\[
   \defcolor{[Z]^T}\ \in\ H^T_{2\dim(Z)}(Y)\,.
\]
The natural map $\frown[Y]^T\colon H^*_T(Y)\to H^T_{2\dim(Y)-*}(Y)$ is an isomorphism when $Y$ is smooth.
Consequently, we may identify $H^T_{-*}(\pt)$ with $S=H^*_T(\pt)$.
We thus have a pairing
\[
   \langle\;,\;\rangle\ \colon\ H^*_T(Y)\otimes H^T_*(Y)\ \longrightarrow\ S\,,
\]
defined by $\langle y,C\rangle:= \rho_*(y\frown C)$.
This satisfies the projection formula; if $\phi\colon Z\to Y$ is a map of projective varieties and we have 
$y\in H^*_T(Y)$ and $C\in H^T_*(Z)$, then 
 \begin{equation}\label{eq:projection_formula}
   \rho_*(\phi^*(y)\frown C)\ =\ \rho_*(y\frown \phi_*(C))\,.\smallskip
 \end{equation}
%

\subsection{Flag varieties}

Let \defcolor{$G$} be a connected and simply connected complex semisimple linear algebraic group,  \defcolor{$B$} a
Borel subgroup of $G$, and \defcolor{$T$} the maximal torus contained in $B$.  
The Weyl group $\defcolor{W}:=N(T)/T$ of $G$ is the quotient of the normalizer of $T$ by $T$. 
Our choice of $B$ gives $W$ the structure of a Coxeter group with a preferred set of 
generators and a length function, $\ell\colon W\to\{0,1,2,\dotsc,\}$. 
These conventions will remain in force for the remainder of this paper.
Let $\defcolor{w_o}$ be the longest element in $W$. 

As any Borel subgroup is its own normalizer and all Borel subgroups are conjugate by elements of $G$, we
may identify the set $\defcolor{\calF}$ of Borel subgroups with the orbit $G/B$, called the \demph{flag manifold}. 
This has a left action by elements of $G$, and we write \defcolor{$g.B$} for the group $gBg^{-1}$.
The inclusion $N(T)\hookrightarrow G$ gives an injection of the Weyl group $W\hookrightarrow\calF$ as 
$T=N(T)\cap B$.
Since for any $w\in W$, the Borel group $w.B$ contains $T$, this identifies the Weyl group 
$W$ with the set \defcolor{$\calF^T$} of $T$-fixed points of $\calF$.

Elements of $W$ also index $B$-orbits on $\calF$, which together form the Bruhat decomposition,  
 \begin{equation}\label{Eq:BruhatDecomposition} 
   \calF\ =\ \bigsqcup_{w\in W} Bw.B/B\,. 
 \end{equation} 
The orbit $Bw.B$ is isomorphic to an affine space of dimension $\ell(w)$ and is a \demph{Schubert cell}. 
Its closure is a \demph{Schubert variety}, \defcolor{$X_w$}. 
Set $\defcolor{B_-}:=w_o.B$, which is the Borel subgroup opposite to $B$ containing $T$. 
Let $\defcolor{X^w}:=\overline{B_-w.B}$, which is also a Schubert variety and has codimension $\ell(w)$. 
The intersection $X^v\cap X_w$ is nonempty if and only if $w\geq v$ and in that case it is irreducible of dimension 
$\ell(w)-\ell(v)$~\cite{Deodhar,Richardson}. 
Note that each of $X^v$, $X_w$, and $X^v\cap X_w$ is $T$-stable.
 

\subsection{Equivariant cohomology of the flag manifold}
We use three presentations for the equivariant cohomology ring $H^*_T(\calF)$ of the flag variety.
Identifying $T$ with $B/[B,B]$, a character $\lambda\in\Xi(T)$ is also a character of $B$.
Write \defcolor{$\C_\lambda$} for the one-dimensional $T$-module $\C$ where $T$ acts via $\lambda$.
The $T$-equivariant line bundle $\defcolor{\calL_\lambda}=G\times_B\,\C_\lambda$ is the
quotient of $G\times\C$ by the equivalence relation $(gb,z)\sim(g,\lambda(b)\cdot z)$ for $g\in G$, $b\in B$, and
$z\in\C$.
Equivariant bundles have equivariant Chern classes.
The map associating a character $\lambda\in\Xi(T)$ to the first equivariant Chern class 
$c_1^T(\calL_\lambda)\in H^2_T(\calF)$ induces a homomorphism of graded algebras, 
$\defcolor{c^T}\colon S\to H^*_T(\calF)$.
The \demph{Borel presentation} of $H^*_T(\calF)$ is the isomorphism
 \[
    S\otimes_{S^W}\! S\ \simeq\ H^*_T(\calF)\,,
\]
which is defined by $f\otimes g\mapsto f\cdot c^T(g)$.
The left copy of $S$ is the pullback $\rho^*(S)$ from a point, and the right copy is the
subring generated by equivariant Chern classes of the $\calL_\lambda$.


Schubert cells are even-dimensional and so the fundamental cycles of Schubert varieties (\demph{Schubert cycles})
form a basis for $H^T_*(\calF)$ over the ring $S$,
\[
     H^T_*(\calF)\ =\ \bigoplus_{w\in W} S\cdot [X_w]^T\,.
\]
There is a dual basis $\{\defcolor{\frakS_v}\mid v\in W\}$ for $H^*_T(\calF)$ with 
$\frakS_v\in H^{2\ell(v)}_T(\calF)$, defined by
\[
   \rho_*( \frakS_v\frown [X_w]^T )\ =\ \delta_{v,w}\,.
\]
Elements of this dual basis are \demph{Schubert classes}, as they are identified with Schubert cycles 
under the isomorphism between equivariant cohomology and equivariant homology, 
 \begin{equation}\label{Eq:duality}
  \frakS_v\frown[\calF]^T\ =\ [X^v]^T
   \qquad\mbox{and}\qquad
  \frakS_v\frown [X_w]^T\ =\ [X^v\cap X_w]^T\,.
 \end{equation}
They lie in the subring generated by the equivariant Chern classes.

There are equivariant \demph{Schubert structure constants} $\defcolor{c^w_{u,v}}\in S$ (of cohomological degree
$2(\ell(u)+\ell(v)-\ell(w))$) defined by the identity 
\[
  \frakS_u \cdot \frakS_v\ =\ \sum_w c^w_{u,v}\, \frakS_w\,.
\]
Graham~\cite{Gr} showed that these are positive sums of monomials in the simple roots of $G$.
Using the duality with Schubert cycles $[X_w]^T$, we have
 \begin{equation}\label{Eq:triple-Int}
   c^w_{u,v}\ =\ \rho_* (\frakS_u\cdot\frakS_v\frown [X_w]^T)\ =\ 
                 \rho_*(\frakS_u\frown [X^v\cap X_w]^T)\,.\smallskip
 \end{equation}


The inclusion $i\colon\calF^T\hookrightarrow\calF$ of the $T$-fixed points into $\calF$ induces the
localization map
 \begin{equation}\label{Eq:localization} 
  i^*\ \colon\ H^*_T(\calF)\ \longrightarrow\ H^*_T(\calF^T)\
   \ =\ \bigoplus_{w\in W} S\,. 
\end{equation} 
Chang and Skjelbred~\cite{ChSk74} describe the image of the localization map and consequently the
full equivariant cohomology ring.
Let \defcolor{$\calF_1$} be the \demph{equivariant one-skeleton} of $\calF$: the set of points
whose stabilizer has codimension at most one in $T$.

\begin{proposition}\label{P:CS}
 Let $j\colon \calF_1\hookrightarrow\calF$ be the inclusion of the equivariant one-skeleton of $\calF$.
 Then the map $i^*\colon H^*_T(\calF)\to H^*_T(\calF^T)$ is an injection and it has the same image as the map 
 $j^*\colon H^*_T(\calF_1)\to H^*_T(\calF^T)$.
\end{proposition}

The flag manifold $\calF$ has finitely many one-dimensional $T$-orbits.
For such a space, Goresky, Kottwitz, and MacPherson~\cite{GKM} used the result of Chang and Skjelbred to give an
elegant description of the image of $\jmath^*$ and thus a description of $H^*_T(\calF)$.

Identify $H^*_T(\calF^T)=\bigoplus_{w\in W} S$ with the set $\Maps(W,S)$ of functions $\phi\colon W\to S$:
If $y\in H^*_T(\calF)$ and $\phi=i^*(y)$ then $\phi(w)$ is defined to be $i^*_w(y)\in S$.
Here, $i_w$ is the map $i_w\colon\pt\to w.B\in\calF^T$.
Classes $\phi\in\Maps(W,S)$ lying in the image of $H^*_T(\calF)$ under $i^*$ satisfy the 
\demph{GKM relations}:
For a root $\alpha$ of $G$ and $u\in W$, we have
 \[
    \phi(u)\ -\ \phi(s_\alpha u)\ \in\ \langle\alpha\rangle\,,
 \]
where $s_\alpha\in W$ is the reflection corresponding to $\alpha$ and $\langle\alpha\rangle$ is the principal ideal
of $S$ generated by $\alpha$  (as roots of $G$ are characters of $T$).

At each $T$-fixed point, there is a two-to-one correspondence between roots $\alpha$ of $G$ and $T$-invariant
curves ($\alpha$ and $-\alpha$ correspond to the same curve whose stabilizer is annihilated by $\alpha$).
Figure~\ref{F:One_Skeleta} displays the equivariant one-skeleta of $Sp(4,\C)/B$ and
$SL(4,\C)/B$, which have Lie types $C_2$ and $A_3$, respectively. 
\begin{figure}[htb]
   \raisebox{25pt}{\includegraphics[height=100pt]{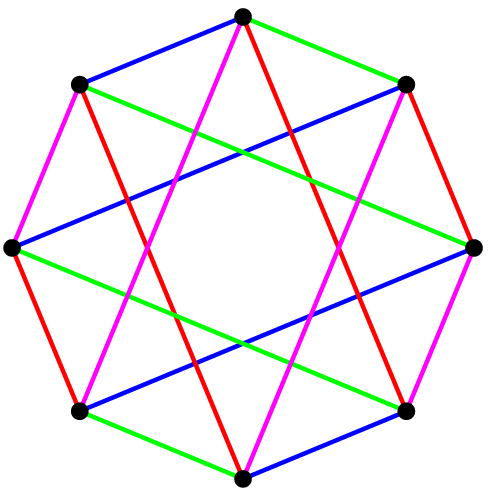}}\qquad
   \includegraphics[height=150pt]{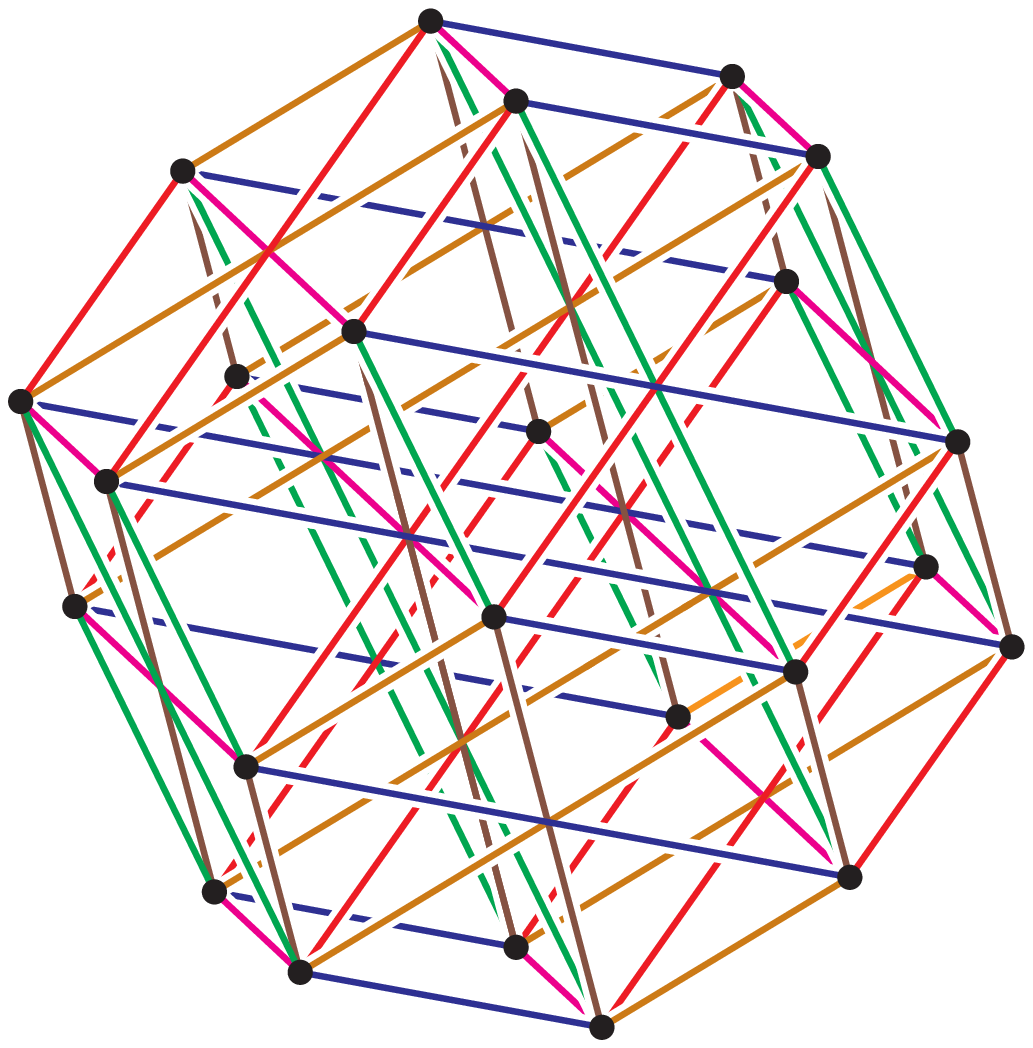}

 \caption{Equivariant one-skeleta.}
 \label{F:One_Skeleta}
\end{figure}

Let us describe the map $i^*\colon H_T^*(\calF)\to H_T^*(\calF^T)$ in more detail.
The inclusion of a $T$-fixed point $w.B$ is a map
\[
   i_w\ \colon\ \pt \ \longrightarrow\ w.B\ \in\ \calF\,.
\]
The Weyl group $W$ acts on characters $\Xi(T)$ of $T$ on the right via
 \begin{equation}\label{Eq:rightAction}
    w.\lambda(t)\ :=\ \lambda(\sigma^{-1} t \sigma)\,,
 \end{equation}
where $\sigma\in N(T)$ is any representative of the right coset $w\subset N(T)$, $T\sigma=w$.
The action~\eqref{Eq:rightAction} on $\Xi(T)$ induces a right action of $W$ on $S$.
The following is standard; we include a proof to illustrate our conventions.

\begin{lemma}\label{L:LB_fixedPoint}
 Let $\lambda\in\Xi(T)$.
 Then $i^*_w(\calL_\lambda)= \C_{w.\lambda}$.
\end{lemma}

\begin{proof}
 The pullback $i^*_w(\calL_\lambda)$ of $\calL_\lambda$ along $i_w$ is $w.B\times_B\C_\lambda$.
 We determine the character of the action of $T$ on 
 $i_w^*(\calL_\lambda)$.
 Let $\sigma\in N(T)$ be a representative of $w$ so that $w=T\sigma$.
 Points of $w.B\times_B\C_\lambda$ have unique representatives $(\sigma,z)$ for $z\in \C$.
 Let $t\in T$.
 Then
\[
     (t\sigma,z)\ =\ (\sigma \sigma^{-1}t\sigma, z)\ =\ 
     (\sigma, \lambda(\sigma^{-1}t\sigma)\cdot z)\,,
\]
 which shows that $w.B\times_B\C_\lambda=\C_{w.\lambda}$.
\end{proof}

\begin{lemma}
 Let $f\otimes g\in S\otimes_{S^W}\! S\simeq H_T^*(\calF)$ and $w\in W$.
 Then
\[
    i_w^*(f\otimes g)\ =\ f\cdot (w.g)\,.
\]
\end{lemma}

\begin{proof}
 As $i^*_w$ is an $S$-module map, 
 $i_w^*(f\otimes g)=f\cdot i_w^*(g)$.
 Since $g$ lies in the subring of $H^*_T(\calF)$ generated by equivariant Chern classes of the 
 $\calL_\lambda$, the result follows by Lemma~\ref{L:LB_fixedPoint}.
\end{proof}

\subsection{Equivariant $K$-theory}
For a projective $T$-variety, let \defcolor{$K^0_T(Y)$} be the Grothendieck ring of $T$-equivariant vector bundles
on $Y$. 
Pullback of vector bundles along a map of $T$-varieties $\phi\colon Y\to Z$ induces the map 
$\phi^*\colon K^0_T(Z)\to K^0_T(Y)$.
The representation ring \defcolor{$R(T)$} of $T$ is $K^0_T(\pt)$, which is the group algebra $\Z[\Xi(T)]$.
The pullback along the map $\rho\colon Y\to \pt$ induces on $K^0_T(Y)$ the structure of
an $R(T)$-algebra.

The Grothendieck group \defcolor{$K^T_0(Y)$} of $T$-equivariant sheaves on $Y$ is a module over $K^0_T(Y)$ via
tensor product.
The alternating sum of higher direct images (derived pushforward) gives a functorial map 
$\phi_*\colon K^T_0(Y)\to K^T_0(Z)$ for any map $\phi\colon Y\to Z$ of projective $T$-varieties.
When $Y$ is smooth, the natural map $K^0_T(Y)\to K^T_0(Y)$ is an isomorphism, and we have a pairing 
\[
    \langle\cdot,\cdot\rangle\ \colon\  K^0_T(Y)\otimes_{R(T)}\! K^0_T(Y)\ \longrightarrow\ R(T)
\]
defined by $\langle \xi,\zeta\rangle:= \rho_*(\xi\cdot\zeta)$.
For any $T$-equivariant sheaf/vector bundle $E$ on $Y$, write \defcolor{$[E]$} for its class in the appropriate
Grothendieck group. 


Consider this for the flag manifold $\calF$.
The ring $K^0_Y(\calF)$ admits a second map $\gamma$ from $R(T)$ induced by the map $\Xi(T)\ni\lambda\mapsto
[\calL_\lambda]$. 
The analog of the Borel presentation for equivariant cohomology is due to McLeod~\cite{McL}.

\begin{proposition}
 The map $f\otimes g\mapsto f\cdot\gamma(g)$ induces an isomorphism
\[
    R(T)\otimes_{R(T)^W}\!R(T)\ \longrightarrow\ K^0_T(\calF)\,.
\]
\end{proposition}

There are several bases of $K^0_T(\calF)$ that come from the Bruhat
decomposition~\eqref{Eq:BruhatDecomposition}; we recommend~\cite{GrKu} for details.
The classes $\{[\calO_{X_w}]\mid w\in W\}$ of structure sheaves of Schubert varieties $X_w=\overline{Bw.B}$ form
an $R(T)$-basis for $K^0_T(\calF)$. 
A different basis is given by the classes $\{[\calO_{X^w}]\mid w\in W\}$ of the Schubert varieties
$X^w=\overline{B_-w.B}$.
Let $\defcolor{\calI_w}$ be the sheaf of $\calO_{X_w}$-ideals defining the complement of the Schubert cell $Bw.B$
in $X_w$, then $\{[\calI_w]\mid w\in W\}$ also forms an $R(T)$-basis for $K^0_T(\calF)$.
The last two are dual bases,
\[
   \langle [\calO_{X^v}]\,,\, [\calI_w] \rangle\ =\ \delta_{v,w}\,.
\]
The ideal sheaves may be expressed in terms of Schubert structure sheaves
 \begin{equation}\label{Eq:ideal_sheaves}
   [\calI_w]\ =\ \sum_{v\leq w} (-1)^{\ell(w)-\ell(v)} [\calO_{X^v}]\,.
 \end{equation}

There are \demph{Schubert structure constants} $\defcolor{b^w_{u,v}}\in R(T)$ defined by the identity
\[
   [\calO_{X^u}]\cdot[\calO_{X^v}]\ =\ \sum_w b^w_{u,v}\, [\calO_{X^w}]\,.
\]
Griffeth and Ram~\cite{GrRa} conjectured that these exhibit a positivity generalizing Graham's positivity for
equivariant cohomology.
This was proven by Anderson, Griffeth, and Miller~\cite{AGM}.
These coefficients have a formula similar to~\eqref{Eq:triple-Int},
 \begin{equation}\label{eq:K-SSC}
    b^w_{u,v}\ =\ \rho_*( [\calO_{X^u}]\cdot[\calO_{X^v}] \cdot [\calI_w])\,.
 \end{equation}

Equivariant $K$-theory of the flag manifold has similar behavior with respect to localization as does the
equivariant cohomology.
The following proposition is due to Vezzosi and Vistoli~\cite{VV} and to Knutson~\cite{Ro}.

\begin{proposition}
 Let $\jmath\colon \calF_1\hookrightarrow\calF$ be the inclusion of the equivariant one-skeleton of $\calF$.
 Then the map $\iota^*\colon K^0_T(\calF)\to K^0_T(\calF^T)$ is an injection and it has the same image as the map 
 $\jmath^*\colon K^0_T(\calF_1)\to K^0_T(\calF^T)$.
\end{proposition}

The Grothendieck ring $K^0_T(\calF^T)$ is also similarly simple,
\[
   K^0_T(\calF^T)\
   \ =\ \bigoplus_{w\in W} R(T)\ =\ \Maps(W,R(T))\,, 
\]
and restriction to a fixed point is similar to that of equivariant cohomology,
\[
    i_w^*(f\otimes g)\ =\ f\cdot (w.g)\,,
\]
where $f,g\in R(T)$ and $f\otimes g\in R(T)\otimes_{R(T)^W}\! R(T)\simeq K^0_T(\calF)$.
 
Classes $\phi\in\Maps(W,R(T))$ lying in the image of $K^0_T(\calF)$ under $i^*$ satisfy analogs of the GKM
relations. 
For a root $\alpha$ of $G$ and $u\in W$, we have
 \[
    \phi(u)\ -\ \phi(s_\alpha u)\ \in\ \langle 1-\alpha\rangle\,,
 \]
where $s_\alpha\in W$ is the reflection corresponding to $\alpha$ and $\langle 1-\alpha\rangle$ is the principal ideal
of $R(T)$ generated by $1-\alpha$  (as roots of $G$ are characters of $T$).

%
\section{Geometry of the pattern map}\label{S:pattern} 
Billey and Braden~\cite{BB} defined the geometric pattern map and developed its main properties.
Let $G,B,T,W$ be a connected and simply connected complex semisimple linear algebraic group, a
Borel subgroup, a maximal torus contained in $B$, and Weyl group as before.
Let $\defcolor{\eta}\colon\C^*\to T$ be a cocharacter whose image is the subgroup \defcolor{$T_\eta$} of $T$. 
Springer~\cite[Theorem 6.4.7]{Sp} showed that the centralizer ${G'}:=Z_G(T_\eta)$ of $T_\eta$ in $G$ 
is a connected, reductive subgroup and $T$ is also a maximal torus of $G'$. 
Also, if $B_0\in\calF$ is a fixed point of $T_\eta$, so that $T_\eta\subset B_0$, then 
$B_0\cap G'$ is a Borel subgroup of $G'$. 

In type $A$, if we have $G=GL(n,\C)$ with $W=S_n$, the symmetric group on $n$ letters, then 
$G'\simeq GL(n_1,\C)\times\dotsb\times GL(n_s,\C)$ with $W'\simeq S_{n_1}\times\dotsb\times S_{n_s}$, where
$n=n_1+\dotsb+n_s$.
In general, $G'$ may be any Levi subgroup of $G$.
For example, any group whose Dynkin diagram is obtained from that of $G$ by deleting some nodes.
 
Set $B':=G'\cap B$.
Let $\Blue{\calF'}:=G'/B'$ be the flag variety of $G'$, and \defcolor{$\calF^\eta$} be the set of 
$T_\eta$-fixed points of $\calF$, which retains an action of $G'$. 
Sending a $T_\eta$-fixed point $B_0\in\calF^{\eta}$ to its intersection with $G'$, $B_0\cap G'$, defines a 
$G'$-equivariant map $\defcolor{\psi}\colon\calF^{\eta}\to\calF'$. 
Restricting to $T$-fixed points gives a map $\psi\colon W\to W'$, where \defcolor{$W'$} is the Weyl group of 
$G'$. 
This is the Billey-Postnikov pattern map~\cite{BiPo} which is the unique map 
$\psi\colon W\to W'$ that is $W'$-equivariant in that $\psi(wx)=w\psi(x)$ for $w\in W'$ and $x\in W$, and 
which respects the Bruhat order in that if $\psi(x)\leq\psi(wx)$ in $W'$ with $w\in W'$
and $x\in W$, then $x\leq wx$ in $W$.  
Billey and Braden use this to deduce that the map $\psi$ is a $G'$-equivariant isomorphism between each 
connected component of $\calF^{\eta}$ with the flag variety $\calF'$, and also that the connected 
components of $\calF^{\eta}$ are in bijection with right cosets  $W'\backslash W$ of $W'$ in $W$. 
 
Observe that $B_-\cap G'=\defcolor{B_-'}$, the Borel group opposite to $B'$ containing $T$. 
Let \defcolor{$\calF^{\eta}_\varsigma$} be the component of $\calF^{\eta}$ corresponding to a coset 
$W'\varsigma$ with $\varsigma\in W'\varsigma$ having minimal length, and let 
$\defcolor{\iota_\varsigma}\colon\calF'\xrightarrow{\sim}\calF^{\eta}_\varsigma$ be the corresponding 
section of the pattern map. 
(This is the unique $G'$-equivariant map sending the $T$-fixed point $e.B'\in\calF'$ to the $T$-fixed 
point $\varsigma.B\in\calF^\eta$.) 
We use a refined result of Billey and Braden.

\begin{proposition}[Theorem~2.3~\cite{AS}]
 Let $W'\varsigma$ be a coset of\/ $W'$ in $W$ with $\varsigma$ of minimal length in $W'\varsigma$ and 
 let $\iota_\varsigma\colon \calF'\to\calF^{\eta}$ be the corresponding section of the pattern map.
 For $w\in W'$ we have 
\[
   \iota_\varsigma(X'_w) \ =\ X^\varsigma\cap X_{w\varsigma}\,.
\]
\end{proposition}
 
\begin{corollary}\label{C:pushforward_formulae}
 For $w\in W'$, we have 
 \begin{eqnarray*}
   \iota_{\varsigma,*}[X'_w]^T& =& [X^\varsigma\cap X_{w\varsigma}]^T\ =\ 
      \frakS_{\varsigma}\frown [X_{w\varsigma}]^T, \\
    \iota_{\varsigma,*}[\calO_{X'_w}]& =& [\calO_{X^\varsigma\cap X_{w\varsigma}}]\ =\ 
   [\calO_{X^\varsigma}]\cdot [\calO_{X_{w\varsigma}}]\,,\quad\mbox{and} \\
    \iota_{\varsigma,*}[\calI_w]& =& [\calO_{X^\varsigma}\otimes\calI_{w\varsigma}]\ =\ 
   [\calO_{X^\varsigma}]\cdot [\calI_{w\varsigma}]\,.
 \end{eqnarray*}
\end{corollary}
 
The second equality for Schubert cycles is~\eqref{Eq:duality}, the second equality for Schubert
structure sheaves is due to Brion~\cite[Lemma 2]{Br02}, and the third line is from Lemma~2.12 of~\cite{AS}. 

In~\cite{AS}, we computed the pullback of equivariant line bundles.

\begin{proposition}[Lemma~2.6~\cite{AS}]\label{P:pullbackLineBunde}
   For $\lambda\in\Xi(T)$, we have $\iota^*_\varsigma(\calL_\lambda)= \calL_{\varsigma.\lambda}$.
\end{proposition}

This is compatible with Lemma~\ref{L:LB_fixedPoint}.
Let $w\in W'$ so that $w.B'\in (\calF')^T$.
The pattern map $\iota_\varsigma$ sends $w.B'$ to $w\varsigma.B$.
We have the commutative diagram
\[
  \begin{picture}(87,37)(-1,1)
    \put(-1,1){$\pt$}
     \put(11,10){\vector(1,1){16}}  \put(8,20){\scriptsize$i_w$}
     \put(12,4){\vector(3,1){60}} \put(38,4){\scriptsize$i_{w\varsigma}$}
     \put(30,25){$\calF'$}
     \put(44,30){\vector(1,0){28}} \put(52,33){\scriptsize$\iota_\varsigma$} 
     \put(75,25){$\calF$}
   \end{picture}
\]
Then by Lemma~\ref{L:LB_fixedPoint}, 
\[
  i^*_w(\iota_\varsigma(\calL_\lambda))\ =\ 
  i^*_w(\calL_{\varsigma.\lambda})\ =\ 
  \C_{w.(\varsigma.\lambda)}\ =\ \C_{w\varsigma.\lambda}\ =\ 
  i^*_{w\varsigma}(\calL_\lambda)\,.
\]

As the pattern map $\iota_\varsigma$ is an isomorphism of $T$-varieties, it is an 
isomorphism
of the equivariant one-skeleta of
$\calF'$ and $\calF^\eta_\varsigma$ and a bijection on $T$-fixed points. 

\begin{example}\label{Ex:six_cosets}
 Suppose that $G=SL(4,\C)$ and $\eta(t)=\mbox{diag}(t,t,t^{-1},t^{-1})$.
 Tnen $\C^4\simeq (\C_1)^2\oplus(\C_{-1})^2$ as a $\C^\times$-module under $T_\eta$.
 Thus $G'=SL(2,\C)\times SL(2,\C)$ with Weyl group $\calS_2\times\calS_2$ in $W=\calS_4$, the symmetric group on 
 four letters, and $\calF'=\P^1\times \P^1$, the product of two projective lines.
 The $T_\eta$-fixed point locus in $\calF$ has six components.
 The equivariant one-skeleton of $\P^1\times\P^1$ is identified with the edges and vertices of a square/diamond.
 Figure~\ref{F:six_cosets} shows the equivariant one-skeleton of $\calF$ with the equivariant one-skeleton of
 $\calF^\eta$ drawn in bold, where the faces corresponding to components of $\calF^\eta$ shaded.
\end{example}
\begin{figure}[htb]
\[
   \includegraphics[height=150pt]{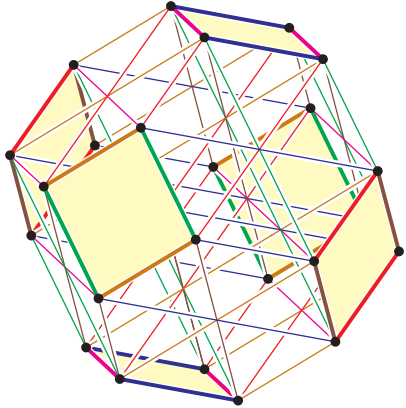}
\]
\caption{Equivariant one-skeleton of $\calF$ and $\calF^\eta$}
\label{F:six_cosets}
\end{figure}

%
\section{The pattern map in equivariant cohomology}\label{S:equivariant} 
 
We give three formulas for the pullback $\iota_\varsigma^*\colon H^*_T(\calF)\to \colon H^*_T(\calF')$
of the pattern map, one for each of our three presentations of equivariant cohomology.

As both $G$ and $G'$ have the same maximal torus $T$, the Borel presentations of equivariant cohomology of $\calF$
and $\calF'$ are nearly identical,
\[
   H^*_T(\calF)\ =\ S\otimes_{S^{W}} S
    \qquad\mbox{and}\qquad
   H^*_T(\calF')\ =\ S\otimes_{S^{W'}} S\,.
\]

\begin{theorem}\label{Th:Borel_Coh}
 Let $\varsigma$ be the minimal length representative of the coset $W'\varsigma$ of\/ $W'$ in $W$ and 
 $\iota_\varsigma\colon\calF'\hookrightarrow\calF$ be the corresponding section of the pattern map.
 The functorial map $\iota^*_\varsigma$ on equivariant cohomology is induced by the map
 $f\otimes g\mapsto f\otimes \varsigma.g$, where $\varsigma$ acts on $S$ through its right action on 
 $\Xi(T)~\eqref{Eq:rightAction}$.
\end{theorem}

\begin{proof}
  Since the left hand copy of $S$ in the Borel presentation of $H^*_T(\calF)$ is simply $\rho^*(S)$, and the same
  for $H^*_T(\calF')$, we must have $\iota^*_\varsigma(f\otimes 1)=f\otimes 1$.

  The right hand copy of $S$ in the Borel presentation of $H^*_T(\calF)$ is generated by the equivariant Chern
  classes of equivariant line bundles $\calL_\lambda$ on $\calF$ and $\calF'$.
  Proposition~\ref{P:pullbackLineBunde} shows that $\iota^*_\varsigma(\calL_\lambda)=\calL_{\varsigma.\lambda}$,
  which implies the theorem. 
\end{proof}

The simplest formula is for localization, for it is essentially restriction of $S$-valued functions.
Let $\iota_\varsigma\colon (\calF')^T\to\calF^T$ be the restriction of the section $\iota_\varsigma$ to the
$T$-fixed points.
For $w\in W'$, we have $\iota_\varsigma(w.B')=w\varsigma.B$.
As the only map between two points is an isomorphism, we have the following.

\begin{proposition}\label{P:EC_local_restriction}
 Let $\phi\in H^*_T(\calF^T)=\Maps(W,S)$.
 Then $\iota^*_\varsigma(\phi)\in H^*_T((\calF')^T)=\Maps(W',S)$ is the map whose value at
 $w\in W'$ is $\phi(w\varsigma)$.
\end{proposition}

Thus if $\alpha\in H^*_T(\calF)$ is represented as a map $\phi\colon W\to S$ via localization,
so that its value at $u\in W$ is $\phi(u)=i^*_u(\alpha)$, then $\iota^*_\varsigma(\alpha)$ is the map 
$W'\to S$ whose value at $w\in W'$ is $i^*_{w\varsigma}(\alpha)=i^*_w(\iota^*_\varsigma(\alpha))$.

We also compute the pattern map in the Schubert basis.

\begin{theorem}\label{Th:EC_Schubert}
 Let $\varsigma$ be the minimal length representative of the coset $W'\varsigma$ of\/ $W'$ in $W$ and 
 $\iota_\varsigma\colon\calF'\hookrightarrow\calF$ be the corresponding section of the pattern map.
 For $u\in W$, we have 
\[
   \iota^*_\varsigma(\frakS_u)\ =\ \sum_{w\in W'} c^{w\varsigma}_{u,\varsigma}\, \frakS_w\,.
\]
\end{theorem}

As the decomposition coefficients expressing the pullback of a Schubert class in the Schubert basis are Schubert
structure constants, they exhibit Graham positivity.

\begin{proof}
 Let $u\in W$.
 As Schubert classes $\frakS_w$ for $w\in W'$ form an $S$-basis of $H^*_T(\calF')$, there are decomposition
 coefficients $d_u^w\in S$ defined by the identity
\[
    \iota^*_\varsigma(\frakS_u)\ =\ \sum_{w\in W'} d_u^w\,\frakS_w\,.
\]
 Using duality and the pushforward map to a point, we have
 \begin{eqnarray*}
   d_u^w&=& \rho_*(\iota^*_\varsigma(\frakS_u)\frown [X_w]^T)\\
        &=& \rho_*(\frakS_u \frown \iota_{\varsigma,*}[X_w]^T)
        \ =\ \rho_*(\frakS_u \frown [X^v\cap X_w]^T)\ =\ c^w_{u,v}\,,
 \end{eqnarray*}
 By the projection formula~\eqref{eq:projection_formula} and~\eqref{Eq:triple-Int}.
\end{proof}

\begin{remark}\label{R:product}
 The formula for $\iota_{\varsigma}^*(\frakS_u)$ in Theorem~\ref{Th:EC_Schubert} gives an algorithm to compute it.
 First expand $\frakS_u\cdot\frakS_\varsigma$ in the Schubert basis of $H^*_T(\calF)$.
 Restrict the sum to terms of the form $\frakS_{w\varsigma}$ with $w\in W'$, and 
 then replace $\frakS_{w\varsigma}$ by $\frakS_{w}$ to obtain the expression for $\iota_{\varsigma}^*(\frakS_u)$.

 As noted following Proposition~\ref{P:EC_local_restriction}, the map $\iota^*_\varsigma$ is particularly simple
 when expressed in terms of localization; it is essentially restriction.
 Theorem~\ref{Th:EC_Schubert} implies that if we restrict a localized Schubert class $i^*(\frakS_u)$ to 
 $W'\varsigma$, considering it as a class in $H^*_T((\calF')^T)$, then it will be a sum of restrictions of Schubert
 classes $\frakS_w$ for $w\in W'$ with coefficients the Schubert structure constants
 $c^{w\varsigma}_{u,\varsigma}$.

 Similarly, if we have a Schubert class $\frakS_u$ expressed as a sum of tensors $f\otimes g$ in the ring
 $S\otimes_{S^W}\! S$, then its pullback $\iota^*_\varsigma(\frakS_u)$ to
 $H^*_T(\calF')=S\otimes_{S^{W'}}\! S$ is the same sum, but where tensors $f\otimes g$  are replaced by 
 $f\otimes(\varsigma.g)$.
 Expanding this in the basis of Schubert classes in $S\otimes_{S^{W'}}\! S$, it will have coefficients the Schubert
 structure constants  $c^{w\varsigma}_{u,\varsigma}$.

 In Section~\ref{S:Ex} we illustrate these interactions between the three formulas for $\iota^*_\varsigma$. 
\end{remark}


%
\section{The pattern map in equivariant $K$-theory}\label{S:K-theory}

We give three formulas for the pullback 
$\iota_\varsigma\colon K^0_T(\calF)\to K^0_T(\calF')$ of the pattern map in equivariant $K$-theory.
The Borel presentations for $K^0_T(\calF)$ and $K^0_T(\calF')$ are nearly identical
\[
   K^0_T(\calF)\ =\ R(T)\otimes_{R(T)^W}\! R(T)
      \qquad\mbox{and}\qquad
  K^0_T(\calF')\ =\ R(T)\otimes_{R(T)^{W'}} R(T)\,.
\]

\begin{theorem} 
  Let $\varsigma$ be the minimal length representative of the coset $W'\varsigma$ of $W'$ in $W$ and
  $\iota_\varsigma\colon\calF'\to\calF$ be the corresponding section of the pattern map.
  The functorial map $\iota^*_\varsigma$ on equivariant $K$-theory is induced by the map
  $f\otimes g\mapsto f\otimes \varsigma.g$, where $W$ acts on $R(T)$ through its right action on 
  $\Xi(T)~\eqref{Eq:rightAction}$.
\end{theorem}

As with equivariant cohomology, the simplest formula, both in terms of statement and proof, is for localized
classes. 

\begin{proposition}
  Let $\phi\in K^0_T(\calF^T)=\Maps(W,R(T))$.
  For $w\in W'$, we have $\iota^*_\varsigma(\phi)(w)=\phi(w\varsigma)$.  
\end{proposition}

We also compute the pattern map in the Schubert basis.

\begin{theorem}
 Let $\varsigma$ be the minimal length representative of the coset $W'\varsigma$ of\/ $W'$ in $W$ and 
 $\iota_\varsigma\colon\calF'\hookrightarrow\calF$ be the corresponding section of the pattern map.
 For $u\in W$, we have 
\[
   \iota^*_\varsigma([\calO^u])\ =\ \sum_{w\in W'} b^{w\varsigma}_{u,\varsigma}\, [\calO^w]\,.
\]
\end{theorem}

These decompositions coefficients expressing the pullback in terms of the Schubert
basis are positive in the sense of Anderson, Griffeth and Miller~\cite{AGM}.

\begin{proof}
 Let $u\in W$.
 As Schubert structure sheaves form a basis for the Grothendieck group, there are decomposition coefficients  
 $d^w_u\in R(T)$ defined by the identity
\[
   \iota^*_\varsigma([\calO_{X^u}])\ =\ \sum_{w\in W'} d^w_u\,[\calO_{X^w}]\,.
\]
 Using the pairing on $K$-theory, the pushforward formula, and the computation in
 Corollary~\ref{C:pushforward_formulae}, we have that
 \begin{eqnarray*}
  d^w_u &=& \rho_*(\iota^*_\varsigma([\calO_{X^u}])\cdot [\calI_w])
       \ =\ \rho_*([\calO_{X^u}]\cdot \iota_{\varsigma,*}([\calI_w]))\\
        &=& \rho_*([\calO_{X^u}]\cdot [\calO_{X^\varsigma}]\cdot [\calI_{w\varsigma}])
       \ =\ b^{w\varsigma}_{u,\varsigma}\,.
 \end{eqnarray*}
 The last equality is~\eqref{eq:K-SSC}.
\end{proof}

%
\section{Examples}\label{S:Ex}

We discuss two examples that illustrate our results for equivariant cohomology.
The first continues Example~2.10 of~\cite{AS}, illustrating Theorem~\ref{Th:EC_Schubert} and
Remark~\ref{R:product}, while
the second illustrates the interplay between all three formulas for equivariant cohomology,
Proposition~\ref{P:EC_local_restriction}, Theorem~\ref{Th:EC_Schubert}, and Theorem~\ref{Th:Borel_Coh}.
All calculations not done explicitly by hand were carried out in Kaji's Maple 
package\fauxfootnote{Available at {\tt http://www.skaji.org/files/bgg\_equivariant.zip}} 
that was released with~\cite{Kaji}.

\begin{example}
 Let $G:=Sp(8,\C)$, the symplectic group of Lie type $C_4$, which is the subgroup of $GL(8,\C)$ 
 preserving the form $\langle x,y\rangle = \sum_{i=1}^4 x_iy_{4+i}-x_{4+i}y_i$.
 Then $g\mapsto \diag(g,(g^T)^{-1})$ embeds $G':=GL(4,\C)$ into $G$ with image the centralizer of the 1-parameter
 subgroup 
\[
   T_\eta\ :=\ 
    \{\diag(t,t,t,t\,,\, t^{-1},t^{-1},t^{-1},t^{-1})
    \mid t\in\C^\times\}\,.
\]
 Both $G$ and $G'$ have the same torus, $T$, which we identify with diagonal $4\times 4$ matrices in $GL(4,\C)$.
 Write $t_1,\dotsc,t_4$ for the standard weights of $T$.

 The Weyl group of $G$ is the group of signed permutations.
 These are words $a_1\,a_2\,a_3\,a_4$, where the absolute values 
 $|a_1|,\dotsc,|a_4|$ form a permutation in $\calS_4$, and the identity element is $1\,2\,3\,4$.
 The length function is
\[
   \ell(a_1\,a_2\,a_3\,a_4)\ =\ \#\{i<j\mid a_i>a_j\}\ +\ \sum_{a_i<0} |a_i|\,.
\]
If we use $\overline{a}$ to represent $-a$, then 
\[
  \ell(3\,\overline{1}\,4\,2)\ =\ 4\,,\quad
  \ell(\overline{2}\,\overline{3}\,4\,1)\ =\ 7\,,\quad
  \mbox{and}\quad  \ell(\overline{2}\,\overline{1}\,3\,4)\ =\ 3\,.
\]

The right cosets of $W'=\calS_4$ are all words obtained by permuting the absolute values without changing signs,
and consequently correspond to subsets $P$ of $\{1,\dotsc,4\}$ indicating the positions of the negative entries.
Here are minimal length coset representatives
\[
   \overline{2}\,\overline{1}\,3\,4\,,\ 
   3\, \overline{2}\,4\,\overline{1}\,,\ 
   2\,3\,\overline{1}\,4\,,\  \mbox{ and }\ 
   \overline{3}\,4\,\overline{2}\,\overline{1}
\]
that correspond to subsets $\{1,2\}$, $\{2,4\}$, $\{3\}$, and $\{1,3,4\}$, respectively.

Write \defcolor{$\frakC_u$} for $u\in C_4$ for equivariant Schubert classes in this type $C_4$ flag manifold $\calF$
and  $\frakS_w$ for $w\in S_4$ for equivariant Schubert classes in the type $A_3$ flag manifold $\calF'$.
Set $\varsigma=\overline{2}\,\overline{1}\,3\,4$ and consider $\iota^*_\varsigma(\frakC_{3\,\overline{1}\,4\,2})$. 
Following Remark~\ref{R:product}, we first compute
$\frakC_{3\,\overline{1}\,4\,2}\cdot\frakC_{\overline{2}\,\overline{1}\,3\,4}$.
%
%
%
 \begin{multline*}
  \frakC_{3\,\overline{1}\,4\,2}\cdot\frakC_{\overline{2}\,\overline{1}\,3\,4}
    \ =\    2(t_1^2{+}t_1t_3)\frakC_{\overline{3}\,\overline{1}\,4\,2}
    \ +\ 2(t_1{+}t_3)\frakC_{\overline{1}\,\overline{3}\,4\,2}
    \ +\ 2t_1\frakC_{\overline{4}\,\overline{1}\,3\,2}
    \ +\ 2(t_1{+}t_2{+}t_3)\frakC_{\overline{3}\,\overline{2}\,4\,1}\\
    \ +\ 2(t_1{+}t_2)\frakC_{3\,\overline{2}\,4\,\overline{1}}
     \ +\ \frakC_{\overline{3}\,\overline{2}\,4\,\overline{1}}
     \ +\   2\frakC_{2\,\overline{3}\,4\,\overline{1}}  \\
     \;\ +\;\   2\frakC_{\overline{4}\,\overline{3}\,1\,2}
     \ +\   2\frakC_{\overline{2}\,\overline{3}\,4\,1}
     \ +\   2\frakC_{\overline{1}\,\overline{4}\,3\,2}
     \ +\   2\frakC_{\overline{4}\,\overline{2}\,3\,1}\,.\qquad
 \end{multline*}
 As only the indices of the first four and last four terms have the form $w\varsigma$, we obtain
\begin{multline*}
  \iota_{\varsigma}^*\bigl(\frakC_{3\,\overline{1}\,4\,2}\bigr)\ =\ 
   2(t_1^2+t_1t_3)\frakS_{1342}
    \ +\ 2(t_1+t_3)\frakS_{3142}
    \ +\ 2t_1\frakS_{1432}
    \ +\ 2(t_1+t_2+t_3)\frakS_{2341}\\
    \ +\ 2\frakS_{3412}
    \ +\ 2\frakS_{3241}
    \ +\ 2\frakS_{4132}
    \ +\ 2\frakS_{2431}\,.
\end{multline*}
The last four terms were computed in Example~2.10 of~\cite{AS} as $\iota^*_\varsigma(\frakC_{3\,\overline{1}\,4\,2})$
in cohomology.
\end{example}

\begin{example}
 Consider the localization formulae for $\iota^*_\varsigma$ when $G$ and $G'$ are as in 
 Example~\ref{Ex:six_cosets}.
 Then $G=SL(4,\C)$, $G'=SL(2,\C)\times SL(2,\C)$ with  $W=\calS_4$ and $W'=\calS_2\times\calS_2$ with generators
 the reflections corresponding to the first and third simple roots $\alpha_{12}:=t_2-t_1$ and $\alpha_{34}:=t_4-t_3$.
 There are six cosets $W'\backslash W$ with minimal representatives
 \begin{equation}\label{Eq:minRtReps}
   \{1234\,,\, 1324\,,\, 3124\,,\, 1342\,,\, 3142 \,,\, 3412\}\,.
 \end{equation}
 Figure~\ref{F:2143} shows the localization $i^*\frakS_{2143}$ displayed with the weak order on $\calS_4$,
\begin{figure}[htb]
  \begin{picture}(270,260)
   \put(0,0){\includegraphics[height=250pt]{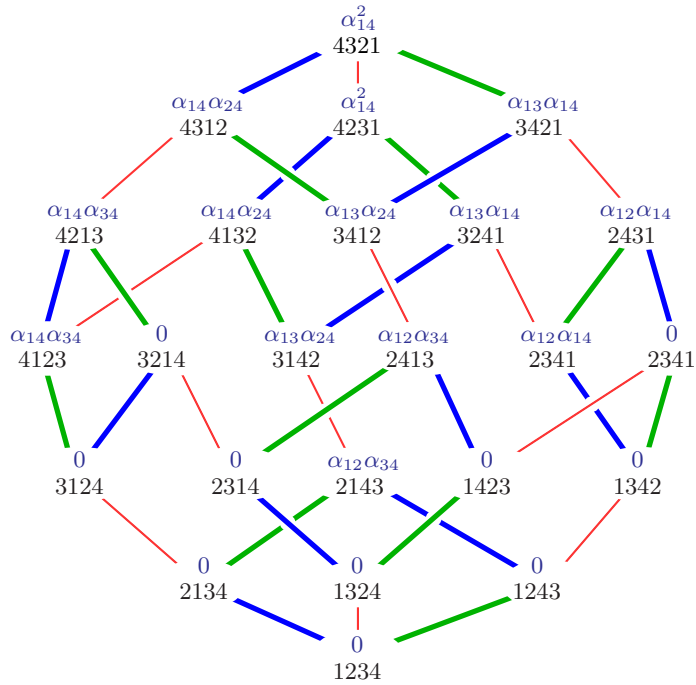}}
   \put(128,237){\scriptsize$4321$} \put(131,247){\defcolor{\scriptsize$\alpha_{14}^2$}}

   \put( 70,206){\scriptsize$4312$} \put( 67,216){\defcolor{\scriptsize$\alpha_{14}\alpha_{24}$}}
   \put(128,206){\scriptsize$4231$} \put(131,216){\defcolor{\scriptsize$\alpha_{14}^2$}}
   \put(197,206){\scriptsize$3421$} \put(194,216){\defcolor{\scriptsize$\alpha_{13}\alpha_{14}$}}

   \put( 23,165){\scriptsize$4213$} \put( 20,175){\defcolor{\scriptsize$\alpha_{14}\alpha_{34}$}}
   \put( 81,165){\scriptsize$4132$} \put( 78,175){\defcolor{\scriptsize$\alpha_{14}\alpha_{24}$}}
   \put(128,165){\scriptsize$3412$} \put(125,175){\defcolor{\scriptsize$\alpha_{13}\alpha_{24}$}}
   \put(175,165){\scriptsize$3241$} \put(172,175){\defcolor{\scriptsize$\alpha_{13}\alpha_{14}$}}
   \put(232,165){\scriptsize$2431$} \put(229,175){\defcolor{\scriptsize$\alpha_{12}\alpha_{14}$}}

   \put(  9,118){\scriptsize$4123$} \put(  6,128){\defcolor{\scriptsize$\alpha_{14}\alpha_{34}$}}
   \put( 54,118){\scriptsize$3214$} \put( 61,128){\defcolor{\scriptsize$0$}}
   \put(105,118){\scriptsize$3142$} \put(102,128){\defcolor{\scriptsize$\alpha_{13}\alpha_{24}$}}
   \put(148,118){\scriptsize$2413$} \put(145,128){\defcolor{\scriptsize$\alpha_{12}\alpha_{34}$}}
   \put(202,118){\scriptsize$2341$} \put(199,128){\defcolor{\scriptsize$\alpha_{12}\alpha_{14}$}}
   \put(247,118){\scriptsize$2341$} \put(254,128){\defcolor{\scriptsize$0$}}

   \put( 23,70){\scriptsize$3124$} \put( 30,80){\defcolor{\scriptsize$0$}}
   \put( 82,70){\scriptsize$2314$} \put( 89,80){\defcolor{\scriptsize$0$}}
   \put(129,70){\scriptsize$2143$} \put(126,80){\defcolor{\scriptsize$\alpha_{12}\alpha_{34}$}}
   \put(177,70){\scriptsize$1423$} \put(184,80){\defcolor{\scriptsize$0$}}
   \put(234,70){\scriptsize$1342$} \put(241,80){\defcolor{\scriptsize$0$}}

   \put( 70,30){\scriptsize$2134$} \put( 77,40){\defcolor{\scriptsize$0$}}
   \put(128,30){\scriptsize$1324$} \put(135,40){\defcolor{\scriptsize$0$}}
   \put(196,30){\scriptsize$1243$} \put(203,40){\defcolor{\scriptsize$0$}}

   \put(128,0){\scriptsize$1234$} \put(135,10){\defcolor{\scriptsize$0$}}
  \end{picture}
\caption{Localization of $\frakS_{2143}$}\label{F:2143}
\end{figure}
which contains the equivariant one-skeleton of $\calF^\eta$.
This is indicated by thick edges with the darker edges
corresponding to the root $\alpha_{12}$ and lighter to $\alpha_{34}$.

Since $\calF'\simeq \P^1\times\P^1$, its equivariant one-skeleton is a diamond.
In Figure~\ref{F:P1P1}, we show five copies of the equivariant one-skeleton, the first labels the $T$-fixed
points, and the remaining four give the localizations $i^*\frakS_w$ for $w=1234,2134,1243, 2143$, in order.
\begin{figure}[htb]
   \begin{picture}(410,81)(-20,-15)
    \put(-20,-5){\begin{picture}(70,65)(-12,-8)
      \put(0,0){\includegraphics{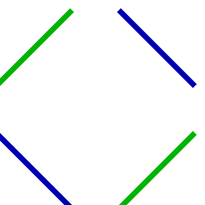}}
                     \put(18,59){\scriptsize{$2143$}}
      \put(-14,25){\scriptsize{$2134$}}        \put(49,25){\scriptsize{$1243$}}
                    \put(18,-9){\scriptsize{$1234$}}
      \end{picture}}
    \put(80,0){\begin{picture}(70,65)(-12,-8)
      \put(0,0){\includegraphics{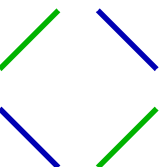}}
                     \put(20,48){\scriptsize{$1$}}
      \put(-8,20){\scriptsize{$1$}}        \put(48,20){\scriptsize{$1$}}
                    \put(20,-8){\scriptsize{$1$}}
                    \put(3,-22){$i^*\frakS_{1234}$}
      \end{picture}}
    \put(160,0){\begin{picture}(70,65)(-12,-8)
      \put(0,0){\includegraphics{pictures/Moment_Graph_A1A1.eps}}
               \put(16,49){\scriptsize{$\alpha_{12}$}}
      \put(-13,21){\scriptsize{$\alpha_{12}$}}        \put(48,20){\scriptsize{$0$}}
                    \put(20,-9){\scriptsize{$0$}}
                    \put(3,-22){$i^*\frakS_{2134}$}
      \end{picture}}
    \put(240,0){\begin{picture}(70,65)(-12,-8)
      \put(0,0){\includegraphics{pictures/Moment_Graph_A1A1.eps}}
                     \put(16,49){\scriptsize{$\alpha_{34}$}}
      \put(-8,20){\scriptsize{$0$}}        \put(40,21){\scriptsize{$\alpha_{34}$}}
                    \put(20,-9){\scriptsize{$0$}}
                    \put(3,-22){$i^*\frakS_{1243}$}
      \end{picture}}
    \put(320,0){\begin{picture}(70,65)(-12,-8)
      \put(0,0){\includegraphics{pictures/Moment_Graph_A1A1.eps}}
          \put(9,49){\scriptsize{$\alpha_{12}\alpha_{34}$}}
      \put(-8,20){\scriptsize{$0$}}        \put(48,20){\scriptsize{$0$}}
                    \put(20,-9){\scriptsize{$0$}}
                    \put(3,-22){$i^*\frakS_{2143}$}
      \end{picture}}
   \end{picture}
\caption{Localizations of equivariant Schubert classes on $\calF'\simeq\P^1\times\P^1$.}
\label{F:P1P1}
\end{figure}

Each of the six components of $\calF^\eta$ gives the localization of a pullback
$\iota^*_\varsigma\frakS_{2143}$ for $\varsigma$ a minimal representative~\eqref{Eq:minRtReps}.
The diamond for $\varsigma=1234$ is simply the localization $i^*\frakS_{2143}$ in Figure~\ref{F:P1P1}.
This agrees with Remark~\ref{R:product}, as $\frakS_\varsigma=1$.
Inspecting the diamond for $\varsigma=1324$, we see that 
$\iota^*_{1324}\frakS_{2143}=\frakS_{2143}$, as $2413=2143\cdot 1324$.
We also compute 
 \begin{equation}\label{Eq:1342}
   \iota^*_{1342}\, i^*\frakS_{2143}\ =\ 
    \raisebox{-28pt}{\begin{picture}(80,65)(-25,-8)
      \put(0,0){\includegraphics{pictures/Moment_Graph_A1A1.eps}}
               \put(9,49){\scriptsize{$\alpha_{12}\alpha_{14}$}}
      \put(-20,21){\scriptsize{$\alpha_{12}\alpha_{14}$}}        \put(48,20){\scriptsize{$0$}}
                    \put(20,-9){\scriptsize{$0$}}
      \end{picture}}
   \ =\ \alpha_{14}\, i^*\frakS_{2134}\,.
 \end{equation}
Most interestingly, $\iota^*_{3412}\, i^*\frakS_{2143}$ is
 \begin{multline*}
    \raisebox{-29pt}{\begin{picture}(83,65)(-20,-8)
      \put(0,0){\includegraphics{pictures/Moment_Graph_A1A1.eps}}
               \put(16,49){\scriptsize{$\alpha_{14}^2$}}
      \put(-20,21){\scriptsize{$\alpha_{13}\alpha_{14}$}} \put(35,21){\scriptsize{$\alpha_{14}\alpha_{24}$}}
                    \put(9,-9){\scriptsize{$\alpha_{13}\alpha_{24}$}}
      \end{picture}}
    \ =\ \alpha_{13}\alpha_{24} \cdot 
       \raisebox{-29pt}{\begin{picture}(68,65)(-10,-8)
      \put(0,0){\includegraphics{pictures/Moment_Graph_A1A1.eps}}
                     \put(20,48){\scriptsize{$1$}}
      \put(-8,20){\scriptsize{$1$}}        \put(48,20){\scriptsize{$1$}}
                    \put(20,-8){\scriptsize{$1$}}
      \end{picture}}
   \ +\ \alpha_{13} \cdot
    \raisebox{-29pt}{\begin{picture}(75,65)(-13,-8)
      \put(0,0){\includegraphics{pictures/Moment_Graph_A1A1.eps}}
               \put(15,49){\scriptsize{$\alpha_{12}$}}
      \put(-13,21){\scriptsize{$\alpha_{12}$}}        \put(48,20){\scriptsize{$0$}}
                    \put(20,-9){\scriptsize{$0$}}
      \end{picture}}
   \ +\ \alpha_{24}\cdot 
    \raisebox{-29pt}{\begin{picture}(70,65)(-10,-8)
      \put(0,0){\includegraphics{pictures/Moment_Graph_A1A1.eps}}
                     \put(16,49){\scriptsize{$\alpha_{34}$}}
      \put(-8,20){\scriptsize{$0$}}        \put(42,21){\scriptsize{$\alpha_{34}$}}
                    \put(20,-9){\scriptsize{$0$}}
      \end{picture}}
   \\
  \ +\ 
    \raisebox{-28pt}{\begin{picture}(64,65)(-8,-8)
      \put(0,0){\includegraphics{pictures/Moment_Graph_A1A1.eps}}
          \put(9,49){\scriptsize{$\alpha_{12}\alpha_{34}$}}
      \put(-8,20){\scriptsize{$0$}}        \put(48,20){\scriptsize{$0$}}
                    \put(20,-9){\scriptsize{$0$}}
      \end{picture}}
    \ \ =\ \alpha_{12}\alpha_{24}\cdot i^*\frakS_{1234}\ +\ 
           \alpha_{13}\cdot i^*\frakS_{2134}\ +\ 
           \alpha_{24}\cdot i^*\frakS_{1243}\ +\ i^*\frakS_{2143}\,,
 \end{multline*}
as $\alpha_{13}\alpha_{24}+\alpha_{13}\alpha_{12}=\alpha_{13}(\alpha_{12}+\alpha_{24})=\alpha_{13}\alpha_{14}$
and similarily, $\alpha_{13}\alpha_{24}+\alpha_{24}\alpha_{34}=\alpha_{14}\alpha_{24}$, and we have
\[
   \alpha_{13}\alpha_{24} + \alpha_{13}\alpha_{12} +\alpha_{24}\alpha_{34}+\alpha_{12}\alpha_{34}
    \ =\ 
   (\alpha_{13}+\alpha_{34})(\alpha_{12}+\alpha_{24})\ =\ \alpha_{14}^2\,.
\]

We compare this computation to the formula of Theorem~\ref{Th:EC_Schubert} as explained in Remark~\ref{R:product}. 
Here are the relevant products,
 \begin{eqnarray*}
  \frakS_{2143}\cdot\frakS_{1324} &=& \frakS_{2413}\ +\ 
                             \Red{\frakS_{4123}+\frakS_{3142}+\frakS_{2341}}\,,\\
  \frakS_{2143}\cdot\frakS_{1342} &=& \alpha_{14}\frakS_{2341}\ \ +\ \ 
                             \Red{\alpha_{24}\frakS_{3142}+\frakS_{3241}}\,,\ \mbox{and}\\
  \frakS_{2143}\cdot\frakS_{3412} &=&  \alpha_{13}\alpha_{24}\frakS_{3412}+ \alpha_{13}\frakS_{3421}
                                     + \alpha_{24}\frakS_{4312} + \frakS_{4321}\,.
 \end{eqnarray*}
 Only the first term in $\frakS_{2143}\frakS_{1324}$ has index of the form $w\cdot 1324$, and for it, $w=2143$, so
 this agrees with the observation that $\iota^*_{1324}\frakS_{2143}=\frakS_{2143}$.
 Similarly, only the first term in $\frakS_{2143}\frakS_{1342}$ has index of the form $w\cdot 1342$, and for it
 $w=2134$ and thus we have $\iota^*_{1342}\frakS_{2143}=\alpha_{14}\frakS_{2134}$, which agrees
 with~\eqref{Eq:1342}.
 Finally, all terms in the product $\frakS_{2143}\frakS_{3412}$ contribute to $\iota^*_{3412}$, and they agree with
 our computation using localization.

 We now look at these same pullbacks in the Borel formulation.
 Let us work with $G=GL(4,\C)$ and $G'=GL(2,\C)\times GL(2,\C)$, and so $T\simeq(\C^\times)^4$ are $4\times 4$
 diagonal matrices (this has no effect on the geometry).
 Write $e_i(x)$ for the degree $i$ elementary symmetric polynomial in its arguments.
 Then we have
\[
    H^*_T(\calF)\ =\ \Q[t_1,\dotsc,t_4,z_1,\dotsc,z_4]/\langle e_i(t)-e_i(z)\mid i=1,\dotsc,4\rangle\,.
\]
 Here $t_i$ are characters of the torus and the equivariant Chern classes generate the image of
 $\Q[z_1,\dotsc,z_4]$.
 Similarly, 
\[
    H^*_T(\calF')\ =\ \Q[t_1,\dotsc,t_4,z_1,\dotsc,z_4]/
     \langle e_i(t_1,t_2)-e_i(z_1,z_2)\,,\, e_i(t_3,t_4)-e_i(z_3,z_4)\mid i=1,\dotsc,2\rangle\,.
\]
 In $H^*_T(\calF')$ we have
\[
   \frakS_{1234}\ =\ 1\,,\ 
   \frakS_{2134}\ =\ z_1{-}t_1\,,\ 
   \frakS_{1243}\ =\ z_3{-}t_3\,,\ \mbox{ and }\ 
   \frakS_{2143}\ =\ (z_1{-}t_1)(z_3{-}t_3)\,.
\]
Observe that $z_1-t_1=t_2-z_2$, $z_3-t_3=t_4-z_4$ and $z_1z_2=t_1t_2$.

We have in $H^*_T(\calF)$,
\[
   \frakS_{2143}\ =\ (z_1-t_1)(z_1+z_2+z_3-t_1-t_2-t_3)\ =\ (z_1-t_1)(t_4-z_4)\,.
\]
Then
 \begin{eqnarray*}
  \iota^*_{1324}\frakS_{2134} &=&(z_1-t_1)(t_4-z_4)\ \ =\ \frakS_{2134}\\
  \iota^*_{1342}\frakS_{2134} &=&(z_1-t_1)(t_4-z_2)\ \ =\ (z_1-t_1)t_4 - z_1z_2+z_2t_1
    \ =\ (z_1-t_1)t_4 - t_1t_2+z_2t_1\\ 
     &=& (z_1-t_1)t_4 - t_1(t_2-z_2)  \ =\ (t_4-t_1)(z_1-t_1) =\ \alpha_{14}\frakS_{2134}\\
  \iota^*_{3412}\frakS_{2134} &=& (z_3-t_1)(t_4-z_2)\ =\ 
    (z_3-t_3 \;+\; t_3-t_1)(t_4-t_2\;+\; t_2-z_2)\\ 
     &=& \bigl((t_3-t_1)+(z_3-t_3)\bigr)\bigl((t_4-t_2)+(z_1-t_1))\bigr)\\ 
     &=&\alpha_{13}\alpha_{24}\frakS_{1234}+ \alpha_{13}\frakS_{2134} 
                      + \alpha_{24}\frakS_{1243} + \frakS_{2143}\,,
 \end{eqnarray*}
which agrees with our previous computations.
\end{example}



\end{document}